\documentclass[reqno]{amsart}
\usepackage{amssymb}
\usepackage{epsfig,amsfonts,amsthm,amssymb,latexsym,amsmath}

\usepackage[a4paper]{geometry}
\begin{document}
\title[Generalized Tribonacci sequence]{A Three-by-Three matrix representation of a generalized Tribonacci sequence}

\author[G. Cerda-Morales]{Gamaliel Cerda-Morales}  

\address{Gamaliel Cerda-Morales \newline
Instituto de Matem\'atica, Pontificia Universidad Cat\'olica de Valpara\'iso, Blanco Viel 596, Cerro Bar\'on, Valpara\'iso, Chile.}
\email{gamaliel.cerda@usm.cl}


\begin{abstract}
The Tribonacci sequence is a well-known example of third order recurrence sequence, which belongs to a particular class of recursive sequences. In this article, other generalized Tribonacci sequence is introduced and defined by $$H_{n+2}=H_{n+1}+H_{n}+H_{n-1}\ \ (n\geq 1),$$ where $H_{0}=3$, $H_{1}=0$ and $H_{2}=2$. Also $n$-th power of the generating matrix for this generalized Tribonacci sequence is established and some basic properties of this sequence are obtained by matrix methods. There are many elementary formulae relating the various $H_{n}$, most of which, since the sequence is defined inductively, are themselves usually proved by induction.
\end{abstract}

\maketitle
\numberwithin{equation}{section}
\newtheorem{theorem}{Theorem}[section]
\newtheorem{lemma}[theorem]{Lemma}
\newtheorem{proposition}[theorem]{Proposition}
\newtheorem{corollary}[theorem]{Corollary}
\newtheorem{definition}[theorem]{Definition}
\newtheorem*{remark}{Remark}

\section{Preliminaries}
Let $Q=\left[
\begin{array}{ccc}
1& 1& 1 \\ 
1& 0 & 0 \\ 
0 & 1& 0
\end{array}
\right]$ be a companion matrix of the Tribonacci sequence $\{T_{n}\}_{n\geq0}$ defined by the third-order linear recurrence relation
\begin{equation}\label{eq:1}
T_{0}=0,\ T_{1}=T_{2}=1,\ \ T_{n}=T_{n-1}+T_{n-2}+T_{n-3}\ (n\geq3).
\end{equation}
Then, by an inductive argument (\cite{Spi}, \cite{Wa}), the $n$-th power $Q^{n}$ has the matrix form
\begin{equation}\label{eq:2}
Q^{n}=\left[
\begin{array}{ccc}
T_{n+1}& T_{n}+T_{n-1}& T_{n} \\ 
T_{n}&T_{n-1}+T_{n-2}& T_{n-1} \\ 
T_{n-1} & T_{n-2}+T_{n-3}& T_{n-2}
\end{array}
\right]\ (n\geq3).
\end{equation}
This property provides an alternate proof of the Cassini-type formula for $\{T_{n}\}_{n\geq0}$:
\begin{equation}\label{eq:3}
T_{n-1}^{3}+T_{n-2}^{2}T_{n+1}+T_{n-3}T_{n}^{2}-2T_{n-2}T_{n-1}T_{n}-T_{n-3}T_{n-1}T_{n+1}=1.
\end{equation}

It is well known \cite{Fe} that the usual Tribonacci numbers can be expressed using Binet's formula
\begin{equation}\label{eq:4}
T_{n}=\frac{\alpha^{n+1}}{(\alpha-\omega_{1})(\alpha-\omega_{2})}-\frac{\omega_{1}^{n+1}}{(\alpha-\omega_{1})(\omega_{1}-\omega_{2})}+\frac{\omega_{2}^{n+1}}{(\alpha-\omega_{2})(\omega_{1}-\omega_{2})}.
\end{equation}
where $\alpha$, $\omega_{1}$ and $\omega_{2}$ are the roots of the cubic equation $x^{3}-x^{2}-x-1=0$. Furthermore, $\alpha=\frac{1}{3}+A_{T}+B_{T}$, $\omega_{1}=\frac{1}{3}+\epsilon A_{T}+\epsilon^{2} B_{T}$ and $\omega_{2}=\frac{1}{3}+\epsilon^{2}A_{T}+\epsilon B_{T}$, where $$A_{T}=\sqrt[3]{\frac{19}{27}+\sqrt{\frac{11}{27}}},\ B_{T}=\sqrt[3]{\frac{19}{27}-\sqrt{\frac{11}{27}}},$$ and $\epsilon=-\frac{1}{2}+\frac{i\sqrt{3}}{2}$ is a primitive cube root of unity.  

From the Binet's formula Eq. (\ref{eq:4}), using the classic identities
\begin{equation}\label{e0}
\alpha+\omega_{1}+\omega_{2}=1,\ \alpha\omega_{1}+\alpha\omega_{2}+\omega_{1}\omega_{2}=-1,\ \alpha\omega_{1}\omega_{2}=1,
\end{equation}
we have for any integer $n\geq2$:
\begin{align*}
\alpha T_{n}&+(1+\omega_{1}\omega_{2})T_{n-1}+T_{n-2}\\
&=\frac{\alpha^{n-1}(\alpha^{3}+(1+\omega_{1}\omega_{2})\alpha+1)}{(\alpha-\omega_{1})(\alpha-\omega_{2})}-\frac{\omega_{1}^{n-1}(\alpha\omega_{1}^{2}+(1+\omega_{1}\omega_{2})\omega_{1}+1)}{(\alpha-\omega_{1})(\omega_{1}-\omega_{2})}\\
&\ \ +\frac{\omega_{2}^{n-1}(\alpha\omega_{2}^{2}+(1+\omega_{1}\omega_{2})\omega_{2}+1)}{(\alpha-\omega_{2})(\omega_{1}-\omega_{2})}\\
&=\alpha^{n}.
\end{align*}
Then, we obtain 
\begin{equation}\label{eq:5}
\alpha T_{n}+(1+\omega_{1}\omega_{2})T_{n-1}+T_{n-2}=\alpha^{n}\ (n\geq2).
\end{equation}
Multipling Eq. (\ref{eq:5}) by $\alpha$, using $\alpha\omega_{1}\omega_{2}=1$, and if we change $\alpha$, $\omega_{1}$ and $\omega_{2}$ role above process, we obtain the quadratic approximation of $\{T_{n}\}_{n\geq0}$
\begin{equation}\label{eq:6}
\textrm{Quadratic app. of $\{T_{n}\}$}: \left\{
\begin{array}{c }
\alpha^{n+1}=T_{n}\alpha^{2}+(T_{n-1}+T_{n-2})\alpha+T_{n-1},\\
\omega_{1}^{n+1}=T_{n}\omega_{1}^{2}+(T_{n-1}+T_{n-2})\omega_{1}+T_{n-1},\\
\omega_{2}^{n+1}=T_{n}\omega_{2}^{2}+(T_{n-1}+T_{n-2})\omega_{2}+T_{n-1},
\end{array}
\right.
\end{equation}
where $\alpha$, $\omega_{1}$ and $\omega_{2}$ are the roots of the cubic equation $x^{3}-x^{2}-x-1=0$.

In Eq. (\ref{eq:6}), if we change $\alpha$, $\omega_{1}$ and $\omega_{2}$ into the companion matrix $Q$ and change $T_{n-1}$ into the matrix $T_{n-1}I$, where $I$ is the $3\times 3$ identity matrix, then we obtain the matrix form Eq. (\ref{eq:2}) of $Q^{n}$
$$Q^{n+1}=T_{n}Q^{2}+(T_{n-1}+T_{n-2})Q+T_{n-1}I\left(=\left[
\begin{array}{ccc}
T_{n+2}& T_{n+1}+T_{n}& T_{n+1} \\ 
T_{n+1}&T_{n}+T_{n-1}& T_{n} \\ 
T_{n} & T_{n-1}+T_{n-2}& T_{n-1}
\end{array}
\right]\right).$$

The Tribonacci sequence has been generalized in many ways, for example, by changing the recurrence relation while preserving the initial terms, by altering the initial terms but maintaining the recurrence relation, by combining of these two techniques, and so on (for more details see \cite{Ge,Pe,Sha,Wa}).

In \cite{Si}, Silvester shows that a numbers of the Fibonacci sequence can be derived from a matrix representation. In \cite{Ko-Bo}, K\"oken and Bozkurt obtained some important properties of Jacobsthal numbers by matrix methods, using diagonalization of a $2\times 2$ matrix to obtain a Binet's formula for the Jacobsthal numbers. In \cite{De}, Demirt\"urk obtained summation formulae for the Fibonacci and Lucas numbers by matrix methods. In \cite{Go-Dha}, Godase and Dhakne described some properties of $k$-Fibonacci and $k$-lucas numbersby matrix terminology. In \cite{Ca,Ca2}, Catarino and Vasco introduced a $2\times 2$ matrix for the $k$-Pell and $k$-Pell-Lucas sequences. This methods give the motivation of our research.

\section{Main Results}
The main aim of this paper is to study other generalized Tribonacci sequence by matrix methods, which is defined below:
\begin{definition}
The generalized Tribonacci sequence, say $\{H_{n}\}_{n\geq 0}$ defined by
\begin{equation}\label{def}
H_{n+2}=H_{n+1}+H_{n}+H_{n-1},\ H_{0}=3,\ H_{1}=0,\ H_{2}=2.
\end{equation}
\end{definition}

Clearly $x^{3}-x^{2}-x-1=0$ is also the characteristic equation of the sequence (\ref{def}). It produces three roots as $\alpha=\frac{1}{3}+A_{T}+B_{T}$, $\omega_{1}=\frac{1}{3}+\epsilon A_{T}+\epsilon^{2} B_{T}$ and $\omega_{2}=\frac{1}{3}+\epsilon^{2}A_{T}+\epsilon B_{T}$, where $$A_{T}=\sqrt[3]{\frac{19}{27}+\sqrt{\frac{11}{27}}},\ B_{T}=\sqrt[3]{\frac{19}{27}-\sqrt{\frac{11}{27}}},$$ and $\epsilon=-\frac{1}{2}+\frac{i\sqrt{3}}{2}$ is a primitive cube root of unity.

Also the $3\times 3$ matrix called generating matrix for the sequence (\ref{def}) is defined as
\begin{equation}\label{mat}
H=\left[
\begin{array}{ccc}
1& 1& 1 \\ 
1& 0 & 0 \\ 
0 & 1& 0
\end{array}
\right].
\end{equation}

\begin{theorem}[Binet Formulae for the Generalized Tribonacci Sequence]\label{t1}
\begin{equation}\label{e1}
\begin{aligned}
H_{n}&=\alpha^{n}+\omega_{1}^{n}+\omega_{2}^{n}-\frac{\alpha^{n+1}}{(\alpha-\omega_{1})(\alpha-\omega_{2})}+\frac{\omega_{1}^{n+1}}{(\alpha-\omega_{1})(\omega_{1}-\omega_{2})}\\
&\ \ -\frac{\omega_{2}^{n+1}}{(\alpha-\omega_{2})(\omega_{1}-\omega_{2})}=K_{n}-T_{n},\ \ n\geq 0,\\
H_{n}&=3T_{n+1}-3T_{n}-T_{n-1},\ \ n\geq 1,
\end{aligned}
\end{equation}
where $K_{n}$ is the $n$-th Tribonacci-Lucas number.
\end{theorem}
\begin{proof}
The general form of the generalized Tribonacci sequence may be expressed in the form:
\begin{equation}\label{e3}
H_{n}=A\alpha^{n}+B\omega_{1}^{n}+C\omega_{2}^{n},
\end{equation}
where $A$, $B$ and $C$ are constants that can be determined by the initial conditions. So put the values $n=0$, $n=1$ and $n=2$ in equation (\ref{e3}), we get $$A+B+C=3,\ A\alpha+B\omega_{1}+C\omega_{2}=0\ \textrm{and}\ A\alpha^{2}+B\omega_{1}^{2}+C\omega_{2}^{2}=2.$$
After solving the above system of equations for $A$, $B$ and $C$, we get $$A=\frac{3\omega_{1}\omega_{2}+2}{(\alpha-\omega_{1})(\alpha-\omega_{2})},\ B=-\frac{3\alpha\omega_{2}+2}{(\alpha-\omega_{1})(\omega_{1}-\omega_{2})},\ C=\frac{3\alpha\omega_{1}+2}{(\alpha-\omega_{2})(\omega_{1}-\omega_{2})}.$$ Therefore,
$$H_{n}=\frac{(3\omega_{1}\omega_{2}+2)\alpha^{n}}{(\alpha-\omega_{1})(\alpha-\omega_{2})}-\frac{(3\alpha\omega_{2}+2)\omega_{1}^{n}}{(\alpha-\omega_{1})(\omega_{1}-\omega_{2})}+\frac{(3\alpha\omega_{1}+2)\omega_{2}^{n}}{(\alpha-\omega_{2})(\omega_{1}-\omega_{2})}$$ and by (\ref{e0}), we have
\begin{equation}\label{e4}
\begin{aligned}
H_{n}&=\frac{((2\omega_{1}\omega_{2}+1)+(1+\omega_{1}\omega_{2}))\alpha^{n}}{(\alpha-\omega_{1})(\alpha-\omega_{2})}-\frac{((2\alpha\omega_{2}+1)+(1+\alpha\omega_{2}))\omega_{1}^{n}}{(\alpha-\omega_{1})(\omega_{1}-\omega_{2})}\\
&\ \ +\frac{((2\alpha\omega_{1}+1)+(1+\alpha\omega_{1}))\omega_{2}^{n}}{(\alpha-\omega_{2})(\omega_{1}-\omega_{2})}\\
&=\frac{((2\omega_{1}\omega_{2}+1)+(\alpha^{2}-\alpha))\alpha^{n}}{(\alpha-\omega_{1})(\alpha-\omega_{2})}-\frac{((2\alpha\omega_{2}+1)+(\omega_{1}^{2}-\omega_{1}))\omega_{1}^{n}}{(\alpha-\omega_{1})(\omega_{1}-\omega_{2})}\\
&\ \ +\frac{((2\alpha\omega_{1}+1)+(\omega_{2}^{2}-\omega_{2}))\omega_{2}^{n}}{(\alpha-\omega_{2})(\omega_{1}-\omega_{2})}\\
&=\frac{((\alpha^{2}+2\omega_{1}\omega_{2}+1)-\alpha)\alpha^{n}}{(\alpha-\omega_{1})(\alpha-\omega_{2})}-\frac{((\omega_{1}^{2}+2\alpha\omega_{2}+1)-\omega_{1})\omega_{1}^{n}}{(\alpha-\omega_{1})(\omega_{1}-\omega_{2})}\\
&\ \ +\frac{((\omega_{2}^{2}+2\alpha\omega_{1}+1)-\omega_{2})\omega_{2}^{n}}{(\alpha-\omega_{2})(\omega_{1}-\omega_{2})}\\
&=\frac{((\alpha-\omega_{1})(\alpha-\omega_{2})-\alpha)\alpha^{n}}{(\alpha-\omega_{1})(\alpha-\omega_{2})}-\frac{((\alpha-\omega_{1})(\omega_{1}-\omega_{2})-\omega_{1})\omega_{1}^{n}}{(\alpha-\omega_{1})(\omega_{1}-\omega_{2})}\\
&\ \ +\frac{((\alpha-\omega_{2})(\omega_{1}-\omega_{2})-\omega_{2})\omega_{2}^{n}}{(\alpha-\omega_{2})(\omega_{1}-\omega_{2})}\\
&=\alpha^{n}+\omega_{1}^{n}+\omega_{2}^{n}-\frac{\alpha^{n+1}}{(\alpha-\omega_{1})(\alpha-\omega_{2})}+\frac{\omega_{1}^{n+1}}{(\alpha-\omega_{1})(\omega_{1}-\omega_{2})}\\
&\ \ -\frac{\omega_{2}^{n+1}}{(\alpha-\omega_{2})(\omega_{1}-\omega_{2})}=K_{n}-T_{n}.
\end{aligned}
\end{equation}
This proves the first part of the theorem (\ref{e1}).

Now if we consider equations (\ref{e0}) and (\ref{e4}), we get
\begin{align*}
H_{n}&=\frac{(3\omega_{1}\omega_{2}+2)\alpha^{n}}{(\alpha-\omega_{1})(\alpha-\omega_{2})}-\frac{(3\alpha\omega_{2}+2)\omega_{1}^{n}}{(\alpha-\omega_{1})(\omega_{1}-\omega_{2})}+\frac{(3\alpha\omega_{1}+2)\omega_{2}^{n}}{(\alpha-\omega_{2})(\omega_{1}-\omega_{2})}\\
&=\frac{(3(\omega_{1}\omega_{2}+1)-1)\alpha^{n}}{(\alpha-\omega_{1})(\alpha-\omega_{2})}-\frac{(3(\alpha\omega_{2}+1)-1)\omega_{1}^{n}}{(\alpha-\omega_{1})(\omega_{1}-\omega_{2})}+\frac{(3(\alpha\omega_{1}+1)-1)\omega_{2}^{n}}{(\alpha-\omega_{2})(\omega_{1}-\omega_{2})}\\
&=\frac{(3(\alpha^{2}-\alpha)-1)\alpha^{n}}{(\alpha-\omega_{1})(\alpha-\omega_{2})}-\frac{(3(\omega_{1}^{2}-\omega_{1})-1)\omega_{1}^{n}}{(\alpha-\omega_{1})(\omega_{1}-\omega_{2})}+\frac{(3(\omega_{2}^{2}-\omega_{2})-1)\omega_{2}^{n}}{(\alpha-\omega_{2})(\omega_{1}-\omega_{2})}\\
&=\frac{(3\alpha^{2}-3\alpha-1)\alpha^{n}}{(\alpha-\omega_{1})(\alpha-\omega_{2})}-\frac{(3\omega_{1}^{2}-3\omega_{1}-1)\omega_{1}^{n}}{(\alpha-\omega_{1})(\omega_{1}-\omega_{2})}+\frac{(3\omega_{2}^{2}-3\omega_{2}-1)\omega_{2}^{n}}{(\alpha-\omega_{2})(\omega_{1}-\omega_{2})}\\
&=3T_{n+1}-3T_{n}-T_{n-1}.
\end{align*}
This proves the second part of the theorem (\ref{e1}).
\end{proof}

\begin{theorem}\label{t2}
For $n\in \mathbb{N}$, we have
\begin{equation}\label{e5}
9H_{n+2}-2H_{n+1}+35H_{n}=41K_{n},
\end{equation}
where $K_{n}$ is the $n$-th Tribonacci-Lucas number.
\end{theorem}
\begin{proof}
To prove this we will use equations (\ref{def}), (\ref{e1}), (\ref{eq:1}) and (\ref{e0}):
\begin{align*}
41K_{n}&=41(H_{n}+T_{n})\\
&=(35H_{n}+6H_{n})+41T_{n}\\
&=35H_{n}+6(3T_{n+1}-3T_{n}-T_{n-1})+41T_{n}\\
&=35H_{n}+18T_{n+1}+23T_{n}-6T_{n-1}\\
&=35H_{n}+18T_{n+1}+23T_{n}-6(T_{n+2}-T_{n+1}-T_{n})\\
&=35H_{n}-6T_{n+2}+24T_{n+1}+29T_{n}\\
&=35H_{n}-6T_{n+2}+24T_{n+1}+27(T_{n+3}-T_{n+2}-T_{n+1})+2T_{n}\\
&=35H_{n}+27T_{n+3}-33T_{n+2}-3T_{n+1}+2T_{n}\\
&=35H_{n}-2(3T_{n+2}-3T_{n+1}-T_{n})+9(3T_{n+3}-3T_{n+2}-T_{n+1})\\
&=35H_{n}-2H_{n+1}+9H_{n+2},
\end{align*}
as required.
\end{proof}

\begin{theorem}[The $n$-th Power of the Generating Matrix]\label{t3}
For $n\in \mathbb{N}$, we have
\begin{equation}\label{e6}
H^{n}=\frac{1}{41}\left[
\begin{array}{ccc}
\begin{array}{c}10H_{n+1}+16H_{n}\\+7H_{n-1}\end{array}& \begin{array}{c}10H_{n}+26H_{n-1}\\
+23H_{n-2}+7H_{n-3}\end{array}& \begin{array}{c}10H_{n}+16H_{n-1}\\+7H_{n-2} \end{array}\\ 
\begin{array}{c}7H_{n+1}+3H_{n}\\+9H_{n-1}\end{array}& \begin{array}{c}7H_{n}+10H_{n-1}\\
+12H_{n-2}+9H_{n-3}\end{array}& \begin{array}{c}7H_{n}+3H_{n-1}\\+9H_{n-2} \end{array}\\ 
\begin{array}{c}9H_{n+1}-2H_{n}\\-6H_{n-1}\end{array}& \begin{array}{c}9H_{n}+7H_{n-1}\\
-8H_{n-2}-6H_{n-3}\end{array}& \begin{array}{c}9H_{n}-2H_{n-1}\\-6H_{n-2} \end{array}
\end{array}
\right],
\end{equation}
where $H_{-1}=-1$, $H_{-2}=-2$.
\end{theorem}
\begin{proof}
Here we shall use induction on $n$. Indeed (\ref{e6}) is true for $n=1$. Now, we suppose that the (\ref{e6}) is true for $n$. Let us show that the result is true for $n+1$, then
\begin{align*}
H^{n+1}&=\frac{1}{41}\left[
\begin{array}{ccc}
\begin{array}{c}10H_{n+1}+16H_{n}\\+7H_{n-1}\end{array}& \begin{array}{c}10H_{n}+26H_{n-1}\\
+23H_{n-2}+7H_{n-3}\end{array}& \begin{array}{c}10H_{n}+16H_{n-1}\\+7H_{n-2} \end{array}\\ 
\begin{array}{c}7H_{n+1}+3H_{n}\\+9H_{n-1}\end{array}& \begin{array}{c}7H_{n}+10H_{n-1}\\
+12H_{n-2}+9H_{n-3}\end{array}& \begin{array}{c}7H_{n}+3H_{n-1}\\+9H_{n-2} \end{array}\\ 
\begin{array}{c}9H_{n+1}-2H_{n}\\-6H_{n-1}\end{array}& \begin{array}{c}9H_{n}+7H_{n-1}\\
-8H_{n-2}-6H_{n-3}\end{array}& \begin{array}{c}9H_{n}-2H_{n-1}\\-6H_{n-2} \end{array}
\end{array}
\right] \left[
\begin{array}{ccc}
1& 1& 1 \\ 
1& 0 & 0 \\ 
0 & 1& 0
\end{array}
\right]\\
&=\frac{1}{41}\left[
\begin{array}{ccc} 
\begin{array}{c}10H_{n+1}+26H_{n}+33H_{n-1}\\+23H_{n-2}+7H_{n-3}\end{array}& \begin{array}{c}10H_{n+1}+26H_{n}\\
+23H_{n-1}+7H_{n-2}\end{array}& \begin{array}{c}10H_{n+1}+16H_{n}\\+7H_{n-1} \end{array}\\ 
\begin{array}{c}7H_{n+1}+10H_{n}+19H_{n-1}\\+12H_{n-2}+9H_{n-3}\end{array}& \begin{array}{c}7H_{n+1}+10H_{n}\\
+12H_{n-1}+9H_{n-2}\end{array}& \begin{array}{c}7H_{n+1}+3H_{n}\\+9H_{n-1} \end{array}\\ 
\begin{array}{c}9H_{n+1}+7H_{n}+H_{n-1}\\-8H_{n-2}-6H_{n-3}\end{array}& \begin{array}{c}9H_{n+1}+7H_{n}\\
-8H_{n-1}-6H_{n-2}\end{array}& \begin{array}{c}9H_{n+1}-2H_{n}\\-6H_{n-1} \end{array}
\end{array}
\right]\\
&=\frac{1}{41}\left[
\begin{array}{ccc} 
\begin{array}{c}10H_{n+2}+16H_{n+1}\\+7H_{n}\end{array}& \begin{array}{c}10H_{n+1}+26H_{n}\\
+23H_{n-1}+7H_{n-2}\end{array}& \begin{array}{c}10H_{n+1}+16H_{n}\\+7H_{n-1} \end{array}\\ 
\begin{array}{c}7H_{n+2}+3H_{n+1}\\+9H_{n}\end{array} & \begin{array}{c}7H_{n+1}+10H_{n}\\
+12H_{n-1}+9H_{n-2}\end{array}& \begin{array}{c}7H_{n+1}+3H_{n}\\+9H_{n-1} \end{array}\\ 
\begin{array}{c}9H_{n+2}-2H_{n+1}\\-6H_{n}\end{array} & \begin{array}{c}9H_{n+1}+7H_{n}\\
-8H_{n-1}-6H_{n-2}\end{array}& \begin{array}{c}9H_{n+1}-2H_{n}\\-6H_{n-1} \end{array}
\end{array}
\right],
\end{align*}
as required.
\end{proof}

\begin{theorem}[Cubic Identity]\label{t4}
For $n\in \mathbb{N}$ with $n\geq 3$, we obtain
\begin{equation}\label{e7}
H_{n-1}^{3}+H_{n-2}^{2}H_{n+1}+H_{n-3}H_{n}^{2}-2H_{n-2}H_{n-1}H_{n}-H_{n-3}H_{n-1}H_{n+1}=41.
\end{equation}
\end{theorem}
\begin{proof}
For Eq. (\ref{mat}), $\det(H^{n})=1$ for all $n\in \mathbb{N}$ and now from Eq. (\ref{e6}), we get
$$H^{n}=\frac{1}{41}\left[
\begin{array}{ccc}
\begin{array}{c}10H_{n+1}+16H_{n}\\+7H_{n-1}\end{array}& \begin{array}{c}10H_{n}+26H_{n-1}\\
+23H_{n-2}+7H_{n-3}\end{array}& \begin{array}{c}10H_{n}+16H_{n-1}\\+7H_{n-2} \end{array}\\ 
\begin{array}{c}7H_{n+1}+3H_{n}\\+9H_{n-1}\end{array}& \begin{array}{c}7H_{n}+10H_{n-1}\\
+12H_{n-2}+9H_{n-3}\end{array}& \begin{array}{c}7H_{n}+3H_{n-1}\\+9H_{n-2} \end{array}\\ 
\begin{array}{c}9H_{n+1}-2H_{n}\\-6H_{n-1}\end{array}& \begin{array}{c}9H_{n}+7H_{n-1}\\
-8H_{n-2}-6H_{n-3}\end{array}& \begin{array}{c}9H_{n}-2H_{n-1}\\-6H_{n-2} \end{array}
\end{array}
\right].$$ Then,
\begin{align*}
\det(H^{n})&=\frac{1}{1681}\left|
\begin{array}{ccc}
\begin{array}{c}10H_{n+1}+16H_{n}\\+7H_{n-1}\end{array}& \begin{array}{c}10H_{n}+26H_{n-1}\\
+23H_{n-2}+7H_{n-3}\end{array}& \begin{array}{c}10H_{n}+16H_{n-1}\\+7H_{n-2} \end{array}\\ 
\begin{array}{c}7H_{n+1}+3H_{n}\\+9H_{n-1}\end{array}& \begin{array}{c}7H_{n}+10H_{n-1}\\
+12H_{n-2}+9H_{n-3}\end{array}& \begin{array}{c}7H_{n}+3H_{n-1}\\+9H_{n-2} \end{array}\\ 
\begin{array}{c}9H_{n+1}-2H_{n}\\-6H_{n-1}\end{array}& \begin{array}{c}9H_{n}+7H_{n-1}\\
-8H_{n-2}-6H_{n-3}\end{array}& \begin{array}{c}9H_{n}-2H_{n-1}\\-6H_{n-2} \end{array}
\end{array}\right| \\
&=\frac{1}{41}\left\lbrace \begin{array}{c} H_{n-1}^{3}+H_{n-2}^{2}H_{n+1}+H_{n-3}H_{n}^{2}\\-2H_{n-2}H_{n-1}H_{n}-H_{n-3}H_{n-1}H_{n+1}\end{array} \right\rbrace .
\end{align*}
Therefore,
$$\left\lbrace \begin{array}{c} H_{n-1}^{3}+H_{n-2}^{2}H_{n+1}+H_{n-3}H_{n}^{2}\\-2H_{n-2}H_{n-1}H_{n}-H_{n-3}H_{n-1}H_{n+1}\end{array} \right\rbrace =41 \det(H^{n}).$$
Since from Eq. (\ref{mat}), $\det(H^{n})=1$. Then,
$$\left\lbrace \begin{array}{c} H_{n-1}^{3}+H_{n-2}^{2}H_{n+1}+H_{n-3}H_{n}^{2}\\-2H_{n-2}H_{n-1}H_{n}-H_{n-3}H_{n-1}H_{n+1}\end{array} \right\rbrace =41.$$ Hence the result.
\end{proof}

\begin{theorem}\label{t5}
For $n\in \mathbb{N}$, we get
\begin{equation}\label{e10}
\left[
\begin{array}{c}
H_{n+2} \\ 
H_{n+1} \\ 
H_{n}
\end{array}
\right]=\left[
\begin{array}{ccc}
1& 1& 1 \\ 
1& 0 & 0 \\ 
0 & 1& 0
\end{array}
\right]\left[
\begin{array}{c}
H_{n+1} \\ 
H_{n} \\ 
H_{n-1}
\end{array}
\right].
\end{equation}
\end{theorem}
\begin{proof}
To prove the ongoing result we shall use induction on $n$. Indeed the result is true for $n=1$. Suppose that the result is true for $n$, then
\begin{align*}
\left[
\begin{array}{c}
H_{n+3} \\ 
H_{n+2} \\ 
H_{n+1}
\end{array}
\right]&=\left[
\begin{array}{c}
H_{n+2}+H_{n+1}+H_{n} \\ 
H_{n+2} \\ 
H_{n+1}
\end{array}
\right]\\
&=\left[
\begin{array}{ccc}
1& 1& 1 \\ 
1& 0 & 0 \\ 
0 & 1& 0
\end{array}
\right]\left[
\begin{array}{c}
H_{n+2} \\ 
H_{n+1} \\ 
H_{n}
\end{array}
\right].
\end{align*}
Since the result is true for $n$ then
\begin{align*}
\left[
\begin{array}{c}
H_{n+3} \\ 
H_{n+2} \\ 
H_{n+1}
\end{array}
\right]&=\left[
\begin{array}{ccc}
1& 1& 1 \\ 
1& 0 & 0 \\ 
0 & 1& 0
\end{array}
\right]^{2}\left[
\begin{array}{c}
H_{n+1} \\ 
H_{n} \\ 
H_{n-1}
\end{array}
\right]\\
&=\left[
\begin{array}{ccc}
1& 1& 1 \\ 
1& 0 & 0 \\ 
0 & 1& 0
\end{array}
\right]\left[
\begin{array}{c}
H_{n+1}+H_{n}+H_{n-1} \\ 
H_{n+1} \\ 
H_{n}
\end{array}
\right]\\
&=\left[
\begin{array}{ccc}
1& 1& 1 \\ 
1& 0 & 0 \\ 
0 & 1& 0
\end{array}
\right]\left[
\begin{array}{c}
H_{n+2} \\ 
H_{n+1} \\ 
H_{n}
\end{array}
\right].
\end{align*}
as desired.
\end{proof}

\begin{theorem}\label{t6}
For $n\geq 0$, we obtain
\begin{equation}\label{e11}
\left[
\begin{array}{c}
H_{n+2} \\ 
H_{n+1} \\ 
H_{n}
\end{array}
\right]=H^{n}\left[
\begin{array}{c}
H_{2} \\ 
H_{1} \\ 
H_{0}
\end{array}
\right].
\end{equation}
\end{theorem}
\begin{proof}
To prove the ongoing result we shall use induction on $n$. Indeed the result is true for $n=0$. Suppose that the result is true for $n$ then $$\left[
\begin{array}{ccc}
1& 1& 1 \\ 
1& 0 & 0 \\ 
0 & 1& 0
\end{array}
\right]^{n+1}\left[
\begin{array}{c}
H_{2} \\ 
H_{1} \\ 
H_{0}
\end{array}
\right]=\left[
\begin{array}{ccc}
1& 1& 1 \\ 
1& 0 & 0 \\ 
0 & 1& 0
\end{array}
\right]H^{n}\left[
\begin{array}{c}
H_{2} \\ 
H_{1} \\ 
H_{0}
\end{array}
\right].$$ Since the result is true for $n$ then
\begin{align*}
\left[
\begin{array}{ccc}
1& 1& 1 \\ 
1& 0 & 0 \\ 
0 & 1& 0
\end{array}
\right]^{n+1}\left[
\begin{array}{c}
H_{2} \\ 
H_{1} \\ 
H_{0}
\end{array}
\right]&=\left[
\begin{array}{ccc}
1& 1& 1 \\ 
1& 0 & 0 \\ 
0 & 1& 0
\end{array}
\right]\left[
\begin{array}{c}
H_{n+2} \\ 
H_{n+1} \\ 
H_{n}
\end{array}
\right]\\
&=\left[
\begin{array}{c}
H_{n+2} +H_{n+1}+H_{n}\\ 
H_{n+2} \\ 
H_{n+1}
\end{array}
\right]\\
&=\left[
\begin{array}{c}
H_{n+3}\\ 
H_{n+2} \\ 
H_{n+1}
\end{array}
\right]
\end{align*}
as desired.
\end{proof}

\section{Binet's Formula by Matrix Diagonalization of Generating Matrix}

In this section we will use the diagonalization of the generating matrix (\ref{mat}) to obtain Binet's formula for the generalized Tribonacci sequence (\ref{def}). For this purpose we should prove the following theorem:

\begin{theorem}\label{t7}
For $n\geq 0$:
\begin{equation}\label{e15}
H_{n}=\frac{1}{\lambda_{T}}\left\lbrace \begin{array}{c} 3(\omega_{1}-\omega_{2})\alpha^{n+2}-3(\alpha-\omega_{2})\omega_{1}^{n+2}+3(\alpha-\omega_{1})\omega_{2}^{n+2}\\ -3(\omega_{1}-\omega_{2})\alpha^{n+1}+3(\alpha-\omega_{2})\omega_{1}^{n+1}-3(\alpha-\omega_{1})\omega_{2}^{n+1} \\ -(\omega_{1}-\omega_{2})\alpha^{n}+(\alpha-\omega_{2})\omega_{1}^{n}-(\alpha-\omega_{1})\omega_{2}^{n}\end{array} \right\rbrace,
\end{equation}
where $\lambda_{T}=(\alpha-\omega_{1})(\alpha-\omega_{2})(\omega_{1}-\omega_{2})$.
\end{theorem}
\begin{proof}
 The generating matrix is given by $H=\left[
\begin{array}{ccc}
1& 1& 1 \\ 
1& 0 & 0 \\ 
0 & 1& 0
\end{array}
\right]$. Now here our motive is to diagonalize the generating matrix $H$. Since $H$ is a square matrix and so let $x$ be the
eigenvalue of $H$ and then by the Cayley Hamilton theorem on matrices, we get 
\begin{align*}
0&=\left| H-xI_{3}\right|\\
&=\left| \left[
\begin{array}{ccc}
1-x& 1& 1 \\ 
1& -x & 0 \\ 
0 & 1& -x
\end{array}
\right]\right|\\
&=x^{3}-x^{2}-x-1.
\end{align*}
This is the characteristic equation of the generating matrix. Let $\alpha$, $\omega_{1}$ and $\omega_{2}$ be the roots of the characteristic equation and also $\alpha$, $\omega_{1}$ and $\omega_{2}$  be the three eigenvalues of the square matrix $H$. Now, we will try to find the eigenvectors corresponding to the eigenvalues $\alpha$, $\omega_{1}$ and $\omega_{2}$. To find the eigenvectors we simply solve the system of linear equations given by
$$(H-xI_{3})v_{x} =0,$$
where $v_{x}$ is the column vector of order $3\times 1$. First of all we calculate the eigenvector corresponding to the eigenvalue $\alpha$ and then
\begin{align*}
0=(H-\alpha I_{3})v_{\alpha}&=\left[
\begin{array}{ccc}
1-\alpha& 1& 1 \\ 
1& -\alpha & 0 \\ 
0 & 1& -\alpha
\end{array}
\right]\left[
\begin{array}{c}
v_{1} \\
v_{2} \\
v_{3}
\end{array}
\right]\\
&=\left[
\begin{array}{c}
(1-\alpha)v_{1}+v_{2}+v_{3} \\ 
v_{1}-\alpha v_{2} \\ 
v_{2}-\alpha v_{3}
\end{array}
\right],
\end{align*}
consider the system
\begin{equation}\label{e12}
\left\{ \begin{array}{lcc}
             (1-\alpha)v_{1}+v_{2}+v_{3}=0\\
             v_{1}-\alpha v_{2}=0 \\
             v_{2}-\alpha v_{3}=0
             \end{array}
   \right.
   \end{equation}
and if we take $v_{3}=c$ in equation (\ref{e12}), we get $v_{2}=\alpha c$ and $v_{1}=\alpha^{2} c$. Hence the eigenvectors corresponding to $\alpha$ are type $\left[
\begin{array}{c}
\alpha^{2} c\\ 
\alpha c\\ 
c
\end{array}
\right]$. In particular $c=1$, the eigenvector corresponding to $\alpha$ is $\left[
\begin{array}{c}
\alpha^{2} \\ 
\alpha \\ 
1
\end{array}
\right]$. Similarly the eigenvectors corresponding to $\omega_{1}$ and $\omega_{2}$ are $\left[
\begin{array}{c}
\omega_{1}^{2} \\ 
\omega_{1} \\ 
1
\end{array}
\right]$ and $\left[
\begin{array}{c}
\omega_{2}^{2} \\ 
\omega_{2} \\ 
1
\end{array}
\right]$, respectively.

Let $P_{T}$ be the matrix of eigenvectors, so $P_{T}=\left[
\begin{array}{ccc}
\alpha^{2}& \omega_{1}^{2}& \omega_{2}^{2} \\ 
\alpha& \omega_{1} & \omega_{2} \\ 
1 & 1& 1
\end{array}
\right]$ and then $$P_{T}^{-1}=\frac{1}{\lambda_{T}}\left[
\begin{array}{ccc}
\omega_{1}-\omega_{2}& -(\omega_{1}+\omega_{2})(\omega_{1}-\omega_{2})& \omega_{1}\omega_{2}(\omega_{1}-\omega_{2})\\ 
-(\alpha-\omega_{2})& (\alpha+\omega_{2})(\alpha-\omega_{2})& -\alpha\omega_{2}(\alpha-\omega_{2})\\ 
\alpha-\omega_{1}& -(\alpha+\omega_{1})(\alpha-\omega_{1})& \alpha\omega_{1}(\alpha-\omega_{1})
\end{array}
\right],$$ where $\lambda_{T}=(\alpha-\omega_{1})(\alpha-\omega_{2})(\omega_{1}-\omega_{2})$.

Now we keep the diagonal matrix $D$ in which eigenvalues of $H$ are on the main diagonal, so $D=\left[
\begin{array}{ccc}
\alpha& 0& 0\\ 
0& \omega_{1}& 0\\ 
0& 0& \omega_{2}
\end{array}
\right]$ and then by the diagonalization of matrices, we get $H=P_{T}DP_{T}^{-1}$. Then,
\begin{align*}
H^{n}&=\left( P_{T}DP_{T}^{-1}\right)^{n}\\
&=P_{T}D^{n}P_{T}^{-1}\\
&=\frac{1}{\lambda_{T}}\left[
\begin{array}{ccc}
\alpha^{2}& \omega_{1}^{2}& \omega_{2}^{2} \\ 
\alpha& \omega_{1} & \omega_{2} \\ 
1 & 1& 1
\end{array}
\right]\left[
\begin{array}{ccc}
\alpha^{n}& 0& 0\\ 
0& \omega_{1}^{n}& 0\\ 
0& 0& \omega_{2}^{n}
\end{array}
\right]\left[
\begin{array}{ccc}
\omega_{1}-\omega_{2}& -(\omega_{1}+\omega_{2})(\omega_{1}-\omega_{2})& \omega_{1}\omega_{2}(\omega_{1}-\omega_{2})\\ 
-(\alpha-\omega_{2})& (\alpha+\omega_{2})(\alpha-\omega_{2})& -\alpha\omega_{2}(\alpha-\omega_{2})\\ 
\alpha-\omega_{1}& -(\alpha+\omega_{1})(\alpha-\omega_{1})& \alpha\omega_{1}(\alpha-\omega_{1})
\end{array}
\right]\\
&=\frac{1}{\lambda_{T}}\left[
\begin{array}{ccc}
\alpha^{n+2}& \omega_{1}^{n+2}& \omega_{2}^{n+2} \\ 
\alpha^{n+1}& \omega_{1}^{n+1}& \omega_{2}^{n+1} \\ 
\alpha^{n}& \omega_{1}^{n}& \omega_{2}^{n} 
\end{array}
\right]\left[
\begin{array}{ccc}
\omega_{1}-\omega_{2}& -(\omega_{1}+\omega_{2})(\omega_{1}-\omega_{2})& \omega_{1}\omega_{2}(\omega_{1}-\omega_{2})\\ 
-(\alpha-\omega_{2})& (\alpha+\omega_{2})(\alpha-\omega_{2})& -\alpha\omega_{2}(\alpha-\omega_{2})\\ 
\alpha-\omega_{1}& -(\alpha+\omega_{1})(\alpha-\omega_{1})& \alpha\omega_{1}(\alpha-\omega_{1})
\end{array}
\right].
\end{align*}
Using the previous equality and Eq. (\ref{e11}), we solve the third row of matrix $H^{n}$:
\begin{align*}
H_{n}&=\frac{1}{\lambda_{T}}\left\lbrace 
2\left\lbrace \begin{array}{c} (\omega_{1}-\omega_{2})\alpha^{n}\\-(\alpha-\omega_{2})\omega_{1}^{n} \\ +(\alpha-\omega_{1})\omega_{2}^{n}\end{array}\right\rbrace +3 \left\lbrace \begin{array}{ccc} \omega_{1}\omega_{2}(\omega_{1}-\omega_{2})\alpha^{n}\\-\alpha \omega_{2}(\alpha-\omega_{2})\omega_{1}^{n} \\ +\alpha\omega_{1}(\alpha-\omega_{1})\omega_{2}^{n}\end{array}\right\rbrace
\right\rbrace \\
&=\frac{1}{\lambda_{T}}\left\lbrace 
\begin{array}{c} (2+3\omega_{1}\omega_{2})(\omega_{1}-\omega_{2})\alpha^{n}\\-(2+3\alpha\omega_{2})(\alpha-\omega_{2})\omega_{1}^{n} \\ +(2+3\alpha\omega_{1})(\alpha-\omega_{1})\omega_{2}^{n}\end{array} 
\right\rbrace \\
&=\frac{1}{\lambda_{T}}\left\lbrace 
\begin{array}{c} (2+3(\alpha^{2}-\alpha-1))(\omega_{1}-\omega_{2})\alpha^{n}\\-(2+3(\omega_{1}^{2}-\omega_{1}-1))(\alpha-\omega_{2})\omega_{1}^{n} \\ +(2+3(\omega_{2}^{2}-\omega_{2}-1))(\alpha-\omega_{1})\omega_{2}^{n}\end{array} 
\right\rbrace \\
&=\frac{1}{\lambda_{T}}\left\lbrace 
\begin{array}{c} (3\alpha^{2}-3\alpha-1)(\omega_{1}-\omega_{2})\alpha^{n}\\-(3\omega_{1}^{2}-3\omega_{1}-1)(\alpha-\omega_{2})\omega_{1}^{n} \\ +(3\omega_{2}^{2}-3\omega_{2}-1)(\alpha-\omega_{1})\omega_{2}^{n}\end{array} 
\right\rbrace \\
&=\frac{1}{\lambda_{T}}\left\lbrace \begin{array}{c} 3(\omega_{1}-\omega_{2})\alpha^{n+2}-3(\alpha-\omega_{2})\omega_{1}^{n+2}+3(\alpha-\omega_{1})\omega_{2}^{n+2}\\ -3(\omega_{1}-\omega_{2})\alpha^{n+1}+3(\alpha-\omega_{2})\omega_{1}^{n+1}-3(\alpha-\omega_{1})\omega_{2}^{n+1} \\ -(\omega_{1}-\omega_{2})\alpha^{n}+(\alpha-\omega_{2})\omega_{1}^{n}-(\alpha-\omega_{1})\omega_{2}^{n}\end{array} \right\rbrace.
\end{align*}
The proof is completed.
\end{proof}
\begin{theorem}\label{t8}
The generalized characteristic roots of $H^{n}$ are
\begin{equation}\label{e16}
\alpha^{n}=\frac{K_{n}}{3}+A_{n}+B_{n},\  \omega_{1}^{n}=\frac{K_{n}}{3}+\epsilon A_{n}+\epsilon^{2} B_{n}\ \textrm{and}\  \omega_{2}^{n}=\frac{K_{n}}{3}+\epsilon^{2}A_{n}+\epsilon B_{n},
\end{equation}
where $$A_{n}=\sqrt[3]{\frac{K_{n}^{3}}{27}-\frac{K_{n}R_{n}}{6}+\frac{1}{2}+\sqrt{\Delta(n)}},\ B_{n}=\sqrt[3]{\frac{K_{n}^{3}}{27}-\frac{K_{n}R_{n}}{6}+\frac{1}{2}-\sqrt{\Delta(n)}},$$ with $\Delta(n)=\frac{K_{n}^{3}}{27}-\frac{K_{n}^{2}R_{n}^{2}}{108}-\frac{K_{n}R_{n}}{6}+\frac{R_{n}^{3}}{27}+\frac{1}{4}$ and $\epsilon=-\frac{1}{2}+\frac{i\sqrt{3}}{2}$.
\end{theorem}
\begin{proof}
If we write. the characteristic polynomial of $H^{n}$, we achieve
\begin{align*}
\left| H^{n}-yI_{3}\right| &=\left|
\begin{array}{ccc}
\frac{\begin{array}{c}10H_{n+1}+16H_{n}\\+7H_{n-1}\end{array}}{41}-y& \frac{\begin{array}{c}10H_{n}+26H_{n-1}\\
+23H_{n-2}+7H_{n-3}\end{array}}{41}& \frac{\begin{array}{c}10H_{n}+16H_{n-1}\\+7H_{n-2} \end{array}}{41}\\ 
\frac{\begin{array}{c}7H_{n+1}+3H_{n}\\+9H_{n-1}\end{array}}{41}& \frac{\begin{array}{c}7H_{n}+10H_{n-1}\\
+12H_{n-2}+9H_{n-3}\end{array}}{41}-y& \frac{\begin{array}{c}7H_{n}+3H_{n-1}\\+9H_{n-2} \end{array}}{41}\\ 
\frac{\begin{array}{c}9H_{n+1}-2H_{n}\\-6H_{n-1}\end{array}}{41}& \frac{\begin{array}{c}9H_{n}+7H_{n-1}\\
-8H_{n-2}-6H_{n-3}\end{array}}{41}& \frac{\begin{array}{c}9H_{n}-2H_{n-1}\\-6H_{n-2} \end{array}}{41}-y
\end{array}\right| \\
&=\frac{1}{1681}\left|
\begin{array}{ccc}
\begin{array}{c}10H_{n+1}+16H_{n}\\+7H_{n-1}-41y\end{array}& \begin{array}{c}10H_{n}+26H_{n-1}\\
+23H_{n-2}+7H_{n-3}\end{array}& \begin{array}{c}10H_{n}+16H_{n-1}\\+7H_{n-2} \end{array}\\ 
\begin{array}{c}7H_{n+1}+3H_{n}\\+9H_{n-1}\end{array}& \begin{array}{c}7H_{n}+10H_{n-1}\\
+12H_{n-2}+9H_{n-3}-41y\end{array}& \begin{array}{c}7H_{n}+3H_{n-1}\\+9H_{n-2} \end{array}\\ 
\begin{array}{c}9H_{n+1}-2H_{n}\\-6H_{n-1}\end{array}& \begin{array}{c}9H_{n}+7H_{n-1}\\
-8H_{n-2}-6H_{n-3}\end{array}& \begin{array}{c}9H_{n}-2H_{n-1}\\-6H_{n-2}-41y \end{array}
\end{array}\right|.
\end{align*}
Then, after using Eqs. (\ref{e5}), (\ref{e6}) and (\ref{e7}) we conclude that
\begin{align*}
\left| H^{n}-yI_{3}\right| &=\left\lbrace \begin{array}{c} H_{n-1}^{3}+H_{n-2}^{2}H_{n+1}+H_{n-3}H_{n}^{2}\\-2H_{n-2}H_{n-1}H_{n}-H_{n-3}H_{n-1}H_{n+1} \\ -y\left\lbrace \begin{array}{c} 4H_{n}H_{n-1}+2H_{n}H_{n-2}+6H_{n}H_{n-3}+6H_{n}^{2}\\-6H_{n+1}H_{n-1}-6H_{n-1}H_{n-2}-4H_{n+1}H_{n-2}\\ +H_{n+1}H_{n-3}-3H_{n-1}^{2}\end{array} \right\rbrace\\ +y^{2}(10H_{n+1}+32H_{n}+15H_{n-1}+6H_{n-2}+9H_{n-3})-41y^{3}
\end{array}\right\rbrace \\
&=\left\lbrace \begin{array}{c} 41-y\left\lbrace \begin{array}{c} 4H_{n}H_{n-1}+2H_{n}H_{n-2}+6H_{n}H_{n-3}+6H_{n}^{2}\\-6H_{n+1}H_{n-1}-6H_{n-1}H_{n-2}-4H_{n+1}H_{n-2}\\ +H_{n+1}H_{n-3}-3H_{n-1}^{2}\end{array} \right\rbrace+41y^{2}K_{n}-41y^{3}
\end{array}\right\rbrace \\
&=41-41yR_{n}+41y^{2}K_{n}-41y^{3}, 
\end{align*}
where $K_{n}=\frac{1}{41}(9H_{n+2}-2H_{n+1}+35H_{n})$ in Eq. (\ref{e5}) and 
\begin{equation}\label{e8}
R_{n}=\frac{1}{41}\left\lbrace \begin{array}{c} 4H_{n}H_{n-1}+2H_{n}H_{n-2}+6H_{n}H_{n-3}+6H_{n}^{2}\\-6H_{n+1}H_{n-1}-6H_{n-1}H_{n-2}-4H_{n+1}H_{n-2}\\ +H_{n+1}H_{n-3}-3H_{n-1}^{2}\end{array} \right\rbrace.
\end{equation}
Hence the characteristic equation of $H^{n}$ is given by
\begin{equation}\label{e9}
y^{3}-K_{n}y^{2}+R_{n}y-1=0
\end{equation}
and the generalized characteristic roots are $y_{1}=\frac{K_{n}}{3}+A_{n}+B_{n}$, $y_{2}=\frac{K_{n}}{3}+\epsilon A_{n}+\epsilon^{2} B_{n}$ and $y_{3}=\frac{K_{n}}{3}+\epsilon^{2}A_{n}+\epsilon B_{n}$, where $$A_{n}=\sqrt[3]{\frac{K_{n}^{3}}{27}-\frac{K_{n}R_{n}}{6}+\frac{1}{2}+\sqrt{\Delta(n)}},\ B_{n}=\sqrt[3]{\frac{K_{n}^{3}}{27}-\frac{K_{n}R_{n}}{6}+\frac{1}{2}-\sqrt{\Delta(n)}},$$ with $\Delta(n)=\frac{K_{n}^{3}}{27}-\frac{K_{n}^{2}R_{n}^{2}}{108}-\frac{K_{n}R_{n}}{6}+\frac{R_{n}^{3}}{27}+\frac{1}{4}$ and $\epsilon=-\frac{1}{2}+\frac{i\sqrt{3}}{2}$. 
Clearly the Eq. (\ref{e9}) has three roots given $\alpha^{n}$, $\omega_{1}^{n}$ and $\omega_{2}^{n}$, and consequently we get the desired result as $\alpha^{n}=\frac{K_{n}}{3}+A_{n}+B_{n}$, $\omega_{1}^{n}=\frac{K_{n}}{3}+\epsilon A_{n}+\epsilon^{2} B_{n}$ and $\omega_{2}^{n}=\frac{K_{n}}{3}+\epsilon^{2}A_{n}+\epsilon B_{n}$. Hence the result.
\end{proof}

\begin{theorem}\label{t9}
The characteristic equation of $H$ is
\begin{equation}\label{m1}
\alpha^{3}-\alpha^{2}-\alpha-1=0.
\end{equation}
\end{theorem}
\begin{proof}
Here we employ the method of matrices as well as determinants to obtain the characteristic equation for $H$. Since
\begin{align*}
H^{n}&=\frac{1}{41}\left[
\begin{array}{ccc}
\begin{array}{c}10H_{n+1}+16H_{n}\\+7H_{n-1}\end{array}& \begin{array}{c}10H_{n}+26H_{n-1}\\
+23H_{n-2}+7H_{n-3}\end{array}& \begin{array}{c}10H_{n}+16H_{n-1}\\+7H_{n-2} \end{array}\\ 
\begin{array}{c}7H_{n+1}+3H_{n}\\+9H_{n-1}\end{array}& \begin{array}{c}7H_{n}+10H_{n-1}\\
+12H_{n-2}+9H_{n-3}\end{array}& \begin{array}{c}7H_{n}+3H_{n-1}\\+9H_{n-2} \end{array}\\ 
\begin{array}{c}9H_{n+1}-2H_{n}\\-6H_{n-1}\end{array}& \begin{array}{c}9H_{n}+7H_{n-1}\\
-8H_{n-2}-6H_{n-3}\end{array}& \begin{array}{c}9H_{n}-2H_{n-1}\\-6H_{n-2} \end{array}
\end{array}\right]\\
\frac{H^{n}}{H_{n-1}}&=\frac{1}{41}\left[
\begin{array}{ccc}
\frac{\begin{array}{c}10H_{n+1}+16H_{n}\\+7H_{n-1}\end{array}}{H_{n-1}}& \frac{\begin{array}{c}10H_{n}+26H_{n-1}\\
+23H_{n-2}+7H_{n-3}\end{array}}{H_{n-1}}& \frac{\begin{array}{c}10H_{n}+16H_{n-1}\\+7H_{n-2} \end{array}}{H_{n-1}}\\ 
\frac{\begin{array}{c}7H_{n+1}+3H_{n}\\+9H_{n-1}\end{array}}{H_{n-1}}& \frac{\begin{array}{c}7H_{n}+10H_{n-1}\\
+12H_{n-2}+9H_{n-3}\end{array}}{H_{n-1}}& \frac{\begin{array}{c}7H_{n}+3H_{n-1}\\+9H_{n-2} \end{array}}{H_{n-1}}\\ 
\frac{\begin{array}{c}9H_{n+1}-2H_{n}\\-6H_{n-1}\end{array}}{H_{n-1}}& \frac{\begin{array}{c}9H_{n}+7H_{n-1}\\
-8H_{n-2}-6H_{n-3}\end{array}}{H_{n-1}}& \frac{\begin{array}{c}9H_{n}-2H_{n-1}\\-6H_{n-2} \end{array}}{H_{n-1}}
\end{array}\right].
\end{align*}
Since the ratio of two consecutive generalized Tribonacci numbers is equal to $\alpha$, then
\begin{align*}
\lim_{n \to \infty}\frac{10H_{n+1}+16H_{n}+7H_{n-1}}{H_{n-1}}&=10 \lim_{n \to \infty}\frac{H_{n+1}}{H_{n-1}}+16 \lim_{n \to \infty}\frac{H_{n}}{H_{n-1}}+7\\
&=10\lim_{n \to \infty}\frac{H_{n+1}}{H_{n}}\cdot\lim_{n \to \infty}\frac{H_{n}}{H_{n-1}}+16 \lim_{n \to \infty}\frac{H_{n}}{H_{n-1}}+7\\
&=10\alpha^{2}+16\alpha +7
\end{align*}
and
\begin{align*}
&\lim_{n \to \infty}\frac{10H_{n}+26H_{n-1}+23H_{n-2}+7H_{n-3}}{H_{n-1}}\\
&=10 \lim_{n \to \infty}\frac{H_{n}}{H_{n-1}}+26 +\frac{23}{\lim_{n \to \infty}\frac{H_{n-1}}{H_{n-2}}}+ \frac{7}{\lim_{n \to \infty}\frac{H_{n-1}}{H_{n-2}}\cdot \lim_{n \to \infty}\frac{H_{n-2}}{H_{n-3}}}\\
&=10 \alpha+26 +\frac{23}{\alpha}+ \frac{7}{\alpha^{2}}.
\end{align*}
Again
\begin{align*}
\lim_{n \to \infty}\frac{10H_{n}+16H_{n-1}+7H_{n-2}}{H_{n-1}}&=10 \lim_{n \to \infty}\frac{H_{n}}{H_{n-1}}+16 +\frac{7}{\lim_{n \to \infty}\frac{H_{n-1}}{H_{n-2}}}\\
&=10\alpha +16+ \frac{7}{\alpha}.
\end{align*}
Therefore,
$$\lim_{n \to \infty}\frac{H^{n}}{H_{n-1}}=\frac{1}{41}\left[
\begin{array}{ccc}
10\alpha^{2}+16\alpha +7& 10 \alpha+26 +\frac{23}{\alpha}+ \frac{7}{\alpha^{2}}& 10\alpha +16+ \frac{7}{\alpha}\\ 
7\alpha^{2}+3\alpha +9& 17 \alpha+10 +\frac{12}{\alpha}+ \frac{9}{\alpha^{2}}& 7\alpha +3+ \frac{9}{\alpha}\\ 
9\alpha^{2}-2\alpha -6& 9 \alpha+7 -\frac{8}{\alpha}- \frac{6}{\alpha^{2}}&9\alpha -2- \frac{6}{\alpha} \end{array}\right].$$
If we consider Eqs. (\ref{e0}) and (\ref{eq:6}), we have
\begin{align*}
&\left[
\begin{array}{ccc}
10\alpha^{2}+16\alpha +7& 10 \alpha+26 +\frac{23}{\alpha}+ \frac{7}{\alpha^{2}}& 10\alpha +16+ \frac{7}{\alpha}\\ 
7\alpha^{2}+3\alpha +9& 17 \alpha+10 +\frac{12}{\alpha}+ \frac{9}{\alpha^{2}}& 7\alpha +3+ \frac{9}{\alpha}\\ 
9\alpha^{2}-2\alpha -6& 9 \alpha+7 -\frac{8}{\alpha}- \frac{6}{\alpha^{2}}&9\alpha -2- \frac{6}{\alpha} \end{array}\right]\\
&=\left[
\begin{array}{ccc}
10\alpha^{2}+16\alpha +7& 16\alpha^{2}+\alpha +3&7\alpha^{2}+3\alpha +9\\ 
7\alpha^{2}+3\alpha +9& 3\alpha^{2}+13\alpha -2& 9\alpha^{2}-2\alpha -6\\ 
9\alpha^{2}-2\alpha -6& -2\alpha^{2}+5\alpha +15& -6\alpha^{2}+15\alpha +4 \end{array}\right].
\end{align*}
If we will compute the determinants of both sides, we get the characteristic equation of the matrix $H$ as follows
\begin{align*}
0&=\left| 
\begin{array}{ccc}
10\alpha^{2}+16\alpha +7& 16\alpha^{2}+\alpha +3&7\alpha^{2}+3\alpha +9\\ 
7\alpha^{2}+3\alpha +9& 3\alpha^{2}+13\alpha -2& 9\alpha^{2}-2\alpha -6\\ 
9\alpha^{2}-2\alpha -6& -2\alpha^{2}+5\alpha +15& -6\alpha^{2}+15\alpha +4 \end{array}\right| \\
&=1681(-\alpha-\alpha^{2}+\alpha^{3}-1)^{2}.
\end{align*}
Then, $\alpha^{3}-\alpha^{2}-\alpha-1=0$ as required.
\end{proof}


\begin{thebibliography}{9}

\bibitem{Ca}
P. Catarino, \emph{A note involving two-by-two matrices of $k$-Pell and $k$-Pell-Lucas sequences}, International Mathematical Forum 8(32) (2013), 1561--1568.
\bibitem{Ca2}
P. Catarino and P. Vasco, \emph{Some basic properties and a two-by-two matrix involving the $k$-Pell numbers}, Int. Journal of Math. Analysis 7(45) (2013), 209--215.
\bibitem{Ce}
G. Cerda-Morales, \emph{Identities for Third Order Jacobsthal Quaternions}, Advances in Applied Clifford Algebras, 27(2) (2017), 1043--1053.
\bibitem{Ce1}
G. Cerda-Morales, \emph{On a Generalization of Tribonacci Quaternions}, Mediterranean Journal of Mathematics 14: 239 (2017), 1--12.
\bibitem{De}
B. Demirt\"urk, \emph{Fibonacci and Lucas sums by matrix methods}, International Mathematical Forum 5(3) (2010), 99--107.
\bibitem{Di}
T.V. Didkivska and M.V. Stopochkina, \emph{Properties of Fibonacci-Narayana numbers}, In the World of Mathematics 9(1)(2003), 29--36.
\bibitem{Fe}
M. Feinberg, \emph{Fibonacci-Tribonacci}, The Fibonacci Quarterly 1(3) (1963), 71--74.
\bibitem{Ge}
W. Gerdes, \emph{Generalized Tribonacci numbers and their convergent sequences}, The Fibonacci Quarterly 16(3) (1978), 269--275.
\bibitem{Go-Dha}
A.D. Godase and M.B. Dhakne, \emph{On the properties of $k$-Fibonacci and $k$-Lucas numbers}, Int. J. Adv. Appl. Math. and Mech. 2(1) (2014), 100--106.
\bibitem{Ki}
E. Kili\c{c}, \emph{Tribonacci sequences with certain indices and their sums}, Ars Comb. 86, (2008), 13--22.
\bibitem{Ko-Bo}
F. K\"oken and D. Bozkurt, \emph{On the Jacobsthal numbers by matrix methos}, Int. J. Contemp. Math. Sciences 3(13) (2008), 605--614.
\bibitem{Ku}
K. Kuhapatanakul and L. Sukruan, \emph{The generalized Tribonacci numbers with negative subscripts}, Integers 14, Paper A32, 6 p. (2014). 
\bibitem{Pe}
S. Pethe, \emph{Some identities for Tribonacci sequences}, The Fibonacci Quarterly 26(2) (1988), 144--151.
\bibitem{Ra}
J.L. Ram\'irez and V.F. Sirvent, \emph{A note on the $k$-Narayana sequence}, Annales Mathematicae et Informaticae 45 (2015), 91--105.
\bibitem{Sha}
A.G. Shannon and A.F. Horadam, \emph{Some properties of third-order recurrence relations}, The Fibonacci Quarterly 10(2) (1972), 135--146.
\bibitem{Si}
J.R. Silverter \emph{Fibonacci properties by matrix methods}, The Mathematical Gazette 63 (1979), 188--191.
\bibitem{Spi}
W. R. Spickerman, \emph{Binet's formula for the Tribonacci numbers}, The Fibonacci Quarterly 20 (1982), 118--120.
\bibitem{Wa}
M.E. Waddill and L. Sacks, \emph{Another generalized Fibonacci sequence}, The Fibonacci Quarterly 5(3) (1967), 209--222.
\bibitem{Ya}
C.C. Yalavigi, \emph{Properties of Tribonacci numbers}, The Fibonacci Quarterly 10(3) (1972), 231--246.
\end{thebibliography}
\end{document}